\documentclass{amsart}

\usepackage{bm}
\usepackage{amssymb}
\usepackage{ifthen}
\usepackage{amsthm}
\usepackage{microtype}
\usepackage{pgfplots}
\usetikzlibrary{pgfplots.groupplots}
\pgfplotsset{compat=1.5}
\usepackage{tikz}
\usepackage{tikz-qtree}
\usetikzlibrary{external}
\usepackage{xparse}
\usepackage{framed}
\usepackage[all]{xy}
\usepackage{theoremref}
\usepackage{comment}
\usepackage{subfigure}
\usepackage{float}
\usepackage{rotating}
\usepackage{changepage}
\usepackage[showframe=false, margin = 1.5     in]{geometry}
\usepackage{paralist}
\usepackage{enumitem}

\newcommand{\tsum}{\textstyle \sum}

\newcommand{\Z}{\mathbb Z}
\newcommand{\tf}{\tfrac}

\newcommand{\inv}{^{-1}}
\newcommand{\f}{\frac}
\newcommand{\ind}[1]{\mathbf{1}{\{#1\}} }

\newcommand{\xto}{\overset{\mathcal X} \to}
\usepackage{color}
\definecolor{dark_green}{RGB}{1, 180, 1}

\usepackage{mathrsfs}
\renewcommand{\l}{\ell}

\newcommand{\xnorm}[2][]{#1 \| #2 #1 \|_{x^{-2}}}
\newcommand{\dnorm}[2][]{#1 \| #2 #1 \|_{x^{-1 - \delta}}}

\theoremstyle{plain}
\newtheorem{thm}{Theorem}
\newtheorem{lemma}[thm]{Lemma}
\newtheorem{prop}[thm]{Proposition}
\newtheorem{cor}[thm]{Corollary}

\theoremstyle{remark}
\newtheorem{remark}[thm]{Remark}

\newcommand{\ep}{\alpha}
\newcommand{\al}{\alpha}

\newcommand{\CC}{\mathscr C}

\DeclareMathOperator{\sgn}{sgn}
\newcommand{\tF}{{\alpha \tilde F^{\Psi}_t}}
\renewcommand{\tilde}{\widetilde}
\DeclareDocumentCommand \D { o }
{%
\IfNoValueTF {#1}
{D}
{\partial_{{#1}}} 
}%
\newcommand{\uniform}{\emph{uniform }}

\begin{document}

\title
{Choices, intervals and equidistribution}

\author{Matthew Junge}
\address{Department of Mathematics, University of Washington}
\email{jungem@math.washington.edu}


\begin{abstract}

We give a sufficient condition for a random sequence in [0,1] generated by a $\Psi$-process to be equidistributed. The condition is met by the canonical example -- the $\max$-2 process -- where the $n$th term is whichever of two uniformly placed points falls in the larger gap formed by the previous $n-1$ points. This solves an open problem from Itai Benjamini, Pascal Maillard and Elliot  Paquette. 
We also deduce equidistribution for more general $\Psi$-processes. This includes an interpolation of the $\min$-2 and $\max$-2 processes that is biased towards $\min$-2. 
\end{abstract}


\maketitle

\section{Introduction}

 A sequence in $[0,1]$ is \emph{equidistributed} if the limiting proportion of points in each subinterval is equal to the subinterval's length.
Over a century ago Weyl proved that $\{ \beta n \mod 1\}_{n \geq 1}$ is equidistributed for any irrational number $\beta$ (see \cite{weyl}). 
Since then connections 
 have been found in ergodic theory, number theory, complex analysis and computer science (\cite{ergodic}, \cite{primes}, \cite{complex}, \cite{computer}). See \cite{uniform} for an overview.
 
Not long after Weyl's Theorem, 
 attention turned to equidistribution of random sequences. 
One way to obtain a random sequence in $[0,1]$ is to independently choose points uniformly. Call the resulting sequence the \emph{uniform process}. The strong law of large numbers guarantees this is equidistributed almost surely. 

Another random process known to equidistribute points is the \emph{Kakutani interval splitting procedure} (introduced in \cite{kakutani0}), where at each step a point is added uniformly to the current largest subinterval. Almost sure equidistribution is proven in \cite{kakutani} and \cite{Loot} using stopping times. Because points are placed in the largest gaps they ought to spread more evenly than the uniform process. Indeed, \cite{pyke} proves the size of the largest interval is asymptotic to $2/n$; the same order as the average interval. Compare to $\log n / n$ in the uniform process (see \cite{uniformsize}). 

\cite{elliot} introduces a family of interval splitting processes that exhibit a wider range of behavior.
The canonical example is the \emph{max-2 process}. The dynamics are as follows:

    \begin{itemize}
        \item Partition $[0,1]$ into subintervals by placing finitely many points in any manner.
        \item At each step sample two points uniformly from $[0,1]$. Each lies in a subinterval formed by the previous configuration. 
        \item Keep the point contained in the larger subinterval and disregard the other point. Break a tie by flipping a fair coin.
    \end{itemize}
    
A discrete analogue of the $\max$-$2$ process appears in \cite{choices1} where $n$ balls are placed into $n$ bins. For each ball two bins are selected uniformly and the ball is placed in the bin with fewer balls. They find that the most-filled bin has $\approx \log_2 \log n$ balls; significantly less then $\approx \log n / \log \log n$ if the balls were instead placed uniformly. This is studied in more detail in \cite{power2} and \cite{choices2}.

In the $\max$-$2$ process choosing the larger gap should spread points more evenly. Despite our intuition this is difficult to formalize, and equidistribution was a primary open problem from \cite{elliot}. 
 The natural counterpart is the \emph{min-2 process} where the point contained in the smaller subinterval is kept. Unlike the previous processes, points are prone to clump together. It is natural to also define the $\max$-$k$ and $\min$-$k$ processes; in these the max or (resp.) min of $k$ candidate points is selected at each step.


Before we can state the theorem we describe a more general splitting procedure known as a \emph{$\Psi$-process} (introduced in \cite{elliot}). For technical convenience we will assume that points arrive according to a Poisson process with intensity $e^t$. Suppose at time $t$ that $N_t$ points have arrived and we have interval lenghts $I_1^{(t)}, I_2^{(t)}, \hdots, I_{N_t}^{(t)}$. Define the size-biased empirical distribution function
$$\tilde A_t(x) = \sum_{i=1}^{N_t} I_i^{(t)} \ind{ I_i^{(t)} \leq x }.$$
This function is now defined to evolve according to Markovian dynamics as follows. Let us say that the next point arrives at time $s>t$, for the $N_s$-th step (with $N_s = N_t+1$) we choose an interval at random, with length $\l_{s} = \tilde A_{s^-}^{-1}(u)$, where $u$ is sampled from a law on $(0,1]$ whose distribution function we denote by $\Psi$. This randomly chosen interval is now subdivided into two pieces at a point chosen uniformly inside the interval. This produces a new sequence of interval lengths $I_1^{(s)}, \hdots I_2^{(s)}, \hdots , I^{(s)}_{ N_s }$ and the process is repeated. Note that $\tilde A_t(x)$ is constant (in $t$) between point arrivals. We remark that the $\max$-$k$, uniform and $\min$-$k$ processes are $\Psi$-processes with $\Psi(u) = u^k, u,$ and $1 -(1-u)^k,$ respectively. 

We abbreviate a few common assumptions for $\Psi$:
\begin{align*}
\text{(C)} & \; \Psi \text{ is continuous.}\\
\text{(C$^1$)} & \; \Psi \text{ is continuously differentiable.}\\
\text{(C$^2$)} & \; \Psi \text{ is twice continuously differentiable.}\\
\text{(D)} & \text{ There exist $c >0$ and $\kappa _\Psi \in [1, \infty)$, such that }  1 - \Psi(u) \geq c(1- u)^{\kappa_\Psi} \text{ for all $u \in (0,1)$.}
\end{align*}
Set $A_t(x) = \tilde A_t(e^{-t}x)$. The main theorem of \cite{elliot} proves that, when (C) and (D) hold, $A_t(x)$ converges pointwise to a (deterministic) continuously differentiable distribution function $F^\Psi(x).$ 
For future theorem statements we note that (C$^1$) and (C$^2$) both imply (D).

Here we study $\tilde A_t^\alpha$, the restriction of $\tilde A_t$ to the $N_t^\alpha$ subintervals contained in $[0,\alpha]$. We find conditions on $\Psi$ that guarantee pointwise convergence $A^\alpha_t \to \alpha F^\Psi$, where $A_t^\alpha(x) = \tilde A_t^\alpha (e^{-t}x)$ and $\alpha F^\Psi$ denotes the map $x \mapsto \alpha \cdot F^\Psi(x)$. When this holds the subinterval lengths in $[0,\alpha]$ evolve to look the same as those in all of [0,1]. This sameness is enough to deduce equidistribution.
\begin{thm}\thlabel{thm:eqd}
Let $\psi = \Psi'$. If $\Psi$ satisfies \emph{(C$^2$)}
 and for some $\delta \in (0,1]$ and all $z \geq 0$
	\begin{align} |z\psi'(F^\Psi(z))  (F^\Psi)'(z) - \psi( F^\Psi(z))|  \leq (2 - \delta) \psi( F^\Psi(z) ) \label{eq:inequality},
\end{align}
	then the $\Psi$-process is equidistributed a.s.\
\end{thm}

The condition \eqref{eq:inequality} arises from a technical computation (see the proof \thref{prop:cond1}) used to show that a family of processes containing $(A_t^\alpha)_{t \geq 0}$ contract in a certain norm. We stress that it is not at all obvious which $\Psi$ and $F^\Psi$ should satisfy this condition. Our only tools are the properties of $F^\Psi$ established in \cite{elliot}. Most importantly, it satisfies the integro-differential equation (see \cite[Lemma 3.5]{elliot}): 
\begin{align} 
(F^\Psi)'(z)  = z \int_z^\infty \f 1 y d \Psi( F^\Psi(y)), \label{eq:integro}
	\end{align}
and the differential equation (see \cite[Proposition 8.1]{elliot}):
\begin{align}
z  (F^\Psi)''(z) -  (F^\Psi)'(z)  + z \psi( F^\Psi(z))  (F^\Psi)'(z) = 0 \label{eq:diff}.
\end{align}
Remarkably, this is enough information to deduce \eqref{eq:inequality} holds for the $\max$-2 process,  an interpolation of $\max$-2 and $\min$-2 processes that is biased towards $\min$-2, and arbitrary interpolations of $\max$-$k$, uniform and $\min$-$k$ processes that place enough weight on the uniform process.  
 
  \begin{cor} \thlabel{cor:main} The following are equidistributed a.s.\
 \begin{enumerate}
 	\item The $\max$-$2$ process.
 	\item The interpolation that is $60\%$-$\min$-$2$ and $40\%$-$\max$-$2$; $\Psi(u) =  .6(1- (1-u)^2) + .4u^2$.
 	\item The interpolation of $\max$-$k$, uniform and $\min$-$k$ processes given by a probability measure $\mathbf p = (p_k)_{k \neq -1,0}$ on $\Z \setminus \{-1,0\}$, that satisfes $\sum_{k \geq 2} k (k-1)[ p_k + p_{-k}] \leq 1/2;$ 
 	$$\Psi(u) = p_1 u + \sum_{k \geq 2} p_k u^k + p_{-k} ( 1 - (1-u)^k).$$ 
 	For example, this includes the interpolations 
 	\begin{enumerate}
 	\item $(1/k^2)\%$-$\min$-$k$ for a single fixed $k$ and otherwise 
        \emph{uniform}.
        \item $99.95\%$-\uniform and $(5^{-k})\%$-$\min$-$k$ for all $k=2,3,\hdots$.
    \end{enumerate}
 \end{enumerate}
 \end{cor}


The reason our approach works for only certain $\Psi$ is unclear. Numerical methods indicate the inequality fails for other processes, suggesting a different approach is needed. This is surprising since processes which ought to better equidistribute points, like a $\max$-3 process, do not meet our criterium. Nonetheless, we conjecture that all $\max$-$k$ and $\min$-$k$ processes are equidistributed. 
The properties established in \thref{master} are an important step in exploring this for $\max$, $\min$ and more general $\Psi$-processes.
The rate of convergence to a uniform placement of points and also the asymptotic size of the largest interval are other important open problems. More thorough discussion can be found in \cite{elliot}.

\subsubsection*{Overview}

This article is organized to quickly arrive at the proof of \thref{thm:eqd}. In Section \ref{sec:prelims} we describe the evolution of intervals in $[0,\alpha]$ and give the major definitions. In Section \ref{sec:thm1} we state without proof \thref{prop:cond1} and \thref{master}. The first proposition describes the importance of \eqref{eq:inequality} holding. The second shows that $A_t^\alpha$ has similar properties as those needed of $A_t$ to deduce convergence in \cite{elliot}. We then use this to establish \thref{thm:eqd}. Section \ref{sec:proofs} contains the proofs for the previous section. 
Finally, in Section \ref{sec:inequalities} we prove \thref{cor:main} by showing that various interpolations satisfy \eqref{eq:inequality}.

\section{Subintervals in $[0,\alpha]$} \label{sec:prelims}

  
We start with a formal definition for a process to be equidistributed. Suppose $n_0$ points are initially placed. After $n$ iterations of an interval splitting process let $N_n^\alpha$ be the number of the first $n_0 + n$ points smaller than $\alpha$. We say a sequence is \emph{equidistributed} if $n\inv N_n^\alpha \to \alpha$ for all $\alpha \in [0,1]$.
It is convenient to work in continuous time. Following \cite{elliot} we have points arrive as a Poisson process with intensity $e^t$. Formal details are in \thref{prop:cond1}. So, in continuous time equidistribution is equivalent to $e^{-t} N_t^\alpha \to \alpha$ for all $\alpha \in [0,1]$.

\subsection{Describing $\tilde {\mathbf A}^\alpha_t$}
 Fix $\alpha \in [0,1]$.  We use the convention that a bold face letter represents a process indexed by time (i.e.\ $\tilde {\mathbf A} = (\tilde A_t)_{t \geq 0}$). 
 Define the joint processes $(\tilde {\mathbf{A}}^\alpha, \tilde{\mathbf A}^{\alpha_+}, \tilde {\mathbf A})$ to be the size-biased empirical distributions of interval lengths contained in $[0,\alpha]$, $[\alpha,1]$ and $[0,1]$, respectively. Formally, letting $I_{1}^{\alpha, (t)},  \hdots, I_{N_t^\alpha}^{\alpha, (t)}$ 
be the lengths of subintervals contained in $[0,\alpha]$ at time $t$ we define
$$\tilde A_t^\alpha(x) = \sum_{j=1}^{N_t^\alpha} I_j^{\alpha, (t)} \cdot \ind{I_j^{\alpha, (t)} \leq x },$$ and similarly for $\tilde A^{\alpha_+}_t$ and $\tilde A_t$. The spark for the refined analysis comes from the relation
\begin{align}
\tilde A_t^{\alpha}(x) + \tilde A_t^{\alpha_+}(x) = \tilde A_t(x), \qquad \forall t,x \geq 0 \label{eq:key}.
\end{align}

To ensure that no intervals are double counted assume the initial set of points placed in $[0,1]$ always contains $\{\alpha\}$. This assumption is only for convenience. Our proof could be adapted to omit it by running the process until two points $\alpha_1 \leq \alpha  \leq \alpha_2$ land sufficiently close to $\alpha$, and then using the bound $N_t^{\alpha_1} \leq N_t^{\alpha} \leq N_t^{\alpha_2}$. We further remark that the same reasoning extends our theorems to the unit circle.

In \cite[Section 2]{elliot} the authors prove that
  $$\tilde A_t(x) = \tilde A_0(x) + \int_0^t e^sx^2 \int_x^\infty \f{\psi(\tilde A_s(z))}{z} d \tilde A_s(z) + \tilde M_t$$
   for some martingale $\tilde M_t$. The following proposition shows that $\tilde A_t^\alpha$ satisfies a similar equation. 
  
  \begin{prop} \thlabel{prop:formula}
Let $\psi = \Psi'$. For any $\Psi$-process satisfying \emph{(C$^1$)}, the joint processes $(\tilde {\mathbf A}^\ep, \tilde {\mathbf A}^{\alpha_+}, \tilde {\mathbf A})$ satisfy the equation
$$\tilde A_t^\ep(x) = \tilde A_0^\ep(x) + \int_0^t e^s x^2 \int_x^\infty \f{ \psi(\tilde A_s(z))}{z} d \tilde A^\ep_s(z) ds + \tilde M^\ep_t(x),$$
with $\tilde M_t^\ep$ a martingale.
\end{prop} 

\begin{proof}
We first build up some necessary definitions. Let $\Psi$ be a continuously differentiable distribution function. 
 Define a Poisson random measure $\prod$ on $[0,\infty) \times [0,1]^2$
with intensity $e^t dt \otimes d \Psi(u) \otimes  dv .$
Set $\l_t(u) = \tilde A_{t^-}\inv(u)$. We use the function $h(v,\l,x) =  v \ind { \l v \leq x } + (1-v) \ind{\l(1-v) \leq x } )$ to ``cut" our sampled interval by $v$. 

We need to detect whether the sampled interval belongs to $[0,\alpha]$. We use the function $g_t^\alpha( \l_t(u)) = \ind{ \l_t(u) \subset  [0,\alpha]}.$ The function $g_t^\alpha$ can be constructed rigorously by assuming all of the subintervals have different lengths, and putting a point mass on each length of subintervals in $[0,\alpha]$. This is a harmless simplification; even for starting configurations with same-length subintervals we know that (when $\Psi \in C^1)$ after an a.s.\ finite time a point will be added to each interval. Once this happens all of the subintervals are of different lengths  a.s.\ and will continue to be of different lengths  a.s.

 We combine all of this to define 
\begin{align*}
\tilde B^\ep(s,u,v,x) &= \l_s(u)\ind{ \l_s(u) > x } g_t^\alpha( \l_s(u)) h(v, \l_s(y)),
\end{align*}
so that
$\tilde A^\alpha_t(x) = \tilde A^\alpha_0(x) + \tsum_{(s,u,v,x)  \in \Pi, s \leq t} \tilde B^\alpha(s,u,v,x).$

Looking to obtain the semimartingale decomposition of $\tilde A^\alpha_t(x)$ we integrate $B(t,u,v,x).$ Note that $\int_0^1 h(v, \l ,x) dv = (x/\l)^2$. We then write

\begin{align*}
	\int \int \tilde B^\alpha(t,u,v,x) dv d \Psi(u)  &= \int_0^1 \l_t(u) \ind{ \l_t(u) > x} g_t^\alpha(\l_s(u)) ( x/ \l_t(u))^2 d \Psi(u) \\
				&= x^2 \int_0^1 \f{ 1 }  { \l_t(u) } \ind{ \l_t(u) > x } g_t^\alpha(\l_t(u))d \Psi(u) \\
				&=x^2 \int_x^\infty \f{ 1 }  { z }  g_t^\alpha(z)d \Psi( \tilde A_{t^-}(z)).
\end{align*}
The last line follows from the fact that for a bounded Borel function, $f$,  $$\int_0^1 f( \l_t(u) ) d \Psi(u) = \int_0^\infty f(z) d \Psi( \tilde A_{t^-}(z) ).$$

Recall that $\Psi$ is assumed to be $C^1$, and that the indicator function $g_t^\alpha$ is zero unless the selected interval belongs to $[0,\alpha]$. This lets us write 
$$g_t^\alpha(z) d \Psi(\tilde A_{t^-}(z)) = \psi( \tilde A_{t^-}(z) ) d \tilde A^\alpha_{t^-}(z).$$
We now rewrite the integral of $\tilde B^\alpha_t$ as
\begin{align*}
		\int \int \tilde B^\alpha(t,u,v,x) dv d \Psi(u) & = x^2 \int_x^\infty \f{ \psi( \tilde A_{t^-}(z) ) }  { z }d\tilde A^\alpha_ {t^-}(z). 
\end{align*}
Integrate this from $0$ to $t$ and we arrive at the claimed decomposition of $\tilde A^\alpha_t(x)$.
\end{proof}


\subsection{Definitions and notation}
What follows are the essential facts and notation for understanding the proof of \thref{thm:eqd}. Let non-tilde processes represent the original process scaled by $e^{-t}$ (i.e.\ $A_t(x) = \tilde A_t(e^{-t} x) )$. In light of \thref{prop:formula}, a change of variables gives the relationship 
\begin{align}\mathbf A^\ep = \CC(\mathbf A^\ep, \mathbf A) + \mathbf M^\ep, \label{eqn:relationship}\end{align}
where $\CC \colon \mathcal X \times \mathcal X \to C( [0,\infty), L^1_{\text{loc}})$  is defined by 
$$\CC(\mathbf F,\mathbf G)_t(x) = F_0(e^{-t} x) + \int_0^t (e^{s-t} x)^2 \int_{ e^{s-t} x} ^\infty \f{ \psi( G_s(z) ) }{z} d F_s(z) ds.$$
Here $\mathcal X= \mathcal B( [0,\infty), \mathcal D)$ 
where $\mathcal D =\{ F \colon [0,\infty) \to [0,1], \text{c\'adl\'ag, increasing}\}$. The set $\mathcal X$ is a subspace of the space $\mathcal B([0,\infty), L_{\text{loc}}^1)$ of measurable maps from $[0,\infty)$ to $L_{\text{loc}}^1$ with the topology of locally uniform convergence, which we denote by the symbol $\overset{\mathcal X} \to$. 

We say that a family of functions $( \mathbf F^{(n)})_{n \in \mathbb N}$ in $\mathcal X$ is \emph{asymptotically equicontinuous} if for every compact $K \subset [0,\infty)$,
$$\lim_{\delta \to 0} \lim_{n \to \infty }  \sup_{\substack{s,t \geq 0 \\ |s-t| \leq \delta} }  \int_K  | F_s^{(n)}(x) - F_s^{(n)}(x) | dx  = 0.$$
A family of distributions $(F_t)_{t \geq 0}$ is \emph{tight} if for all $\epsilon>0$ there exists $N$ such that $F_t(N) \geq 1- \epsilon$ for all $t \geq 0$.

We will use $\hat F$ and $F^{\Psi}$ interchangeably to denote the a.s.\ pointwise limiting distribution of $A_t$ from \cite[Theorem 1.1]{elliot}. Also define the stationary distribution $ \mathbf{\hat F}^*$ so that $\hat F^*_t = \hat F$ for all $t \geq 0$. With the convergence $A_t \to  \hat F$ in mind, we consider the operator
$$\CC^* (\mathbf F)_t = \CC(\mathbf F, \hat{\mathbf F}^*)_t = F_0(e^{-t} x) + \int_0^t (e^{s-t} x)^2 \int_{ e^{s-t} x} ^\infty \f{ \psi(\hat F(z))}{z} d F_s(z) ds.$$
We will see in the proof of \thref{thm:eqd} that the limiting distribution of $A^\alpha_t$ belongs to the set of fixed points
$$\mathfrak F^\alpha = \{ \mathbf F \in \mathcal X_1 \colon \mathbf F = \CC^*(\mathbf F), F_t(+\infty) = \alpha \text{ and } ( \tf 1 \ep F_t)_{t \geq 0} \text{ tight} \}.$$
Here $\mathcal X_1 = \mathcal B( [0,\infty), \{F \in \mathcal D\colon \xnorm{F} \leq 1\})$, where $\xnorm{\cdot}$ is the case $\delta =1$ of  
the following family of norms on $L^1_{\text{loc}}([0,\infty))$:
\begin{align}
\quad\dnorm{f} = \int_0^\infty x^{-1 - \delta} |f(x)| dx, \quad \delta \in (0,1]. \label{eq:norm}
 \end{align} 
 The norm used exclusively in \cite{elliot} is $\xnorm{f} = \int_0^\infty x^{-2} |f(x)| dx $. This extra $\delta$ of freedom lets us prove the interpolation between $\min$-2 and $\max$-2 is equidistributed. The effect of working in this norm is the appearance of the $(2-\delta)$ term in \eqref{eq:inequality}.
 
We remark that $\xnorm{\cdot}$ does have special significance. A key property (see \thref{master} \ref{rate}) is  that $\xnorm{ \tilde A_t^\alpha } = e^{-t} N_t^\alpha.$ Thus, we can recover the number of points added to the interval $[0,\alpha]$, which is the fundamental quantity for proving equidistribution.  
 
 

\section{Proof of \thref{thm:eqd}} \label{sec:thm1}

We delay the proofs of the following two propositions until the next section. Our goal is to make transparent the necessary ingredients for proving \thref{thm:eqd}. 
The first proposition describes the benefit of when a $\Psi$-process satisfies \eqref{eq:inequality}. 
\begin{prop} \thlabel{prop:cond1}
If $\Psi$ satisfies $\emph(\text{\emph{C}}^1)$ and there exists $\delta \in (0,1]$ such that \eqref{eq:inequality} holds for all $z \geq 0$, then $$\dnorm{F_t - \alpha\hat F} \leq 2(1+ \delta \inv) e^{-\delta t }$$ for all $\mathbf F \in \mathfrak F^{\alpha}$.
\end{prop}

\noindent We will also need several general properties of $\mathbf A^\alpha$.

\begin{prop}\thlabel{master} The following hold for any $\Psi$ satisfying $(\text{\emph{C}}^2)$:
\begin{enumerate}[label = {(\Roman*)}, labelindent = .2 cm] 
    \item $\xnorm{A^\al_t} = e^{-t} N^\al_t$ and $\xnorm{\alpha \hat F} = \alpha.$ \label{rate}
    \item The collection of distribution functions $(\f 1 \ep A^\ep_t)_{t \geq 0}$ is tight.  \label{tight} 
    \item The family $( \mathbf A^{\alpha, (n)})$ defined by $A_t^{\alpha, (n)} = A^\alpha_{t+n}$ is asymptotically equicontinuous.  \label{equi} 
    \item $\mathbf M^{\ep, (n)} \overset{\mathcal X}\to 0$ as $n \to \infty$, where $M_t^{\ep,(n)}(x) = M^\ep_{t+n}(x) - M^\ep_n(e^{-t}x)$ for every $t \geq 0.$ \label{noise} 
    \item Suppose additionally that $\sup_{ z \geq 0 }z \hat F'(z) < \infty$ (discussion of this hypothesis appears in \thref{lem:extra}). Define $\mathbf A^{(n)}$ by $A^{(n)}_t = A_{t+n}$. If $\mathbf F^{(n)}\xto \mathbf F$ then $\mathscr C(\mathbf F^{(n)}, \mathbf A^{(n)}) \overset{\mathcal X} \to \mathscr C^*( \mathbf F)$. \label{continuous} 

\end{enumerate}

\end{prop}

\begin{proof}[Proof of \thref{thm:eqd}]
All statements are meant to hold almost surely. Also we abbreviate items from \thref{master} as a roman numeral. 
In the continuous process points are added as a Poisson process with intensity $e^tdt$. So, it suffices to show $e^{-t} N_t^\alpha \to \alpha$. 


 By \ref{tight}, \ref{equi} and the version of the Arzel\'a-Ascoli theorem in \cite[Lemma 7.3]{elliot} we may choose a sequence $(\mathbf A^{\alpha, (n_k)})$ which converges to a family of (scaled by $\alpha$) distributions $\mathbf F^{\alpha, (\infty)}$ with $F_t^{\alpha, (\infty))}(+ \infty ) = \alpha$ for every $t \geq 0$. Taking limits in the formula at \eqref{eqn:relationship} we obtain
$$\CC(\mathbf A^{\ep, (n_k)}, \mathbf A^{(n_k) } ) + \mathbf M^{\alpha , (n_k)} \overset{ \mathcal X} \to \mathbf F^{\ep, (\infty)}.$$
By \ref{noise} and \ref{continuous} we have
$$\CC( \mathbf A^{\ep, (n_k)}, \mathbf A^{(n_k) } )\overset{\mathcal X} \to \CC^*(\mathbf F^{\ep, (\infty) }).$$
Thus, $\mathbf F^{\ep, (\infty) } \in \mathfrak F^\ep$. Since we are assuming \eqref{eq:inequality} holds, \thref{prop:cond1} implies that $ \dnorm{F_t^{\ep, (\infty) } - \alpha\hat F } \leq (2+ \delta \inv) e^{- \delta t}$. A similar argument as the conclusion of the proof of \cite[Theorem 7.1]{elliot} gives almost sure pointwise convergence $A_t^\alpha \to \alpha \hat F$. \cite[Theorem 1.1]{elliot} states that $A_t \to \hat F$ pointwise. We can then deduce from \eqref{eq:key} that $A^{\alpha_+}_t \to (1- \alpha) \hat F$. Combining pointwise convergence, $\eqref{eq:key}$ and Fatou's lemma we deduce that $\xnorm{A_t^\alpha} \to \xnorm{ \alpha \hat F}$. Indeed,
\begin{align*}
\liminf \xnorm{A_t^\alpha} &\geq \xnorm{\alpha \hat F},\\
\limsup \xnorm{A_t^\alpha} &= 1 - \liminf \xnorm{A^{\alpha_+}_t} \leq   1- (1-\alpha) = \xnorm{\alpha \hat F}.
\end{align*}
This finishes the proof since \ref{rate} states that  $\xnorm{A_t^\alpha} = e^{-t} N_t^\alpha$ and $\xnorm{\alpha \hat F}=\alpha$. 
\end{proof}

\section{Proof of \thref{prop:cond1} and \thref{master}}\label{sec:proofs}


\subsection{\thref{prop:cond1}}

The proof of \thref{prop:cond1} proceeds analogously to \cite[Lemma 4.1 and Proposition 3.4]{elliot}. A significant difference is that they apply integration by parts to $$\f 1 z d \Psi(\tilde F_s(z) ),$$ whereas our operator $\mathscr C^*$ requires applying integration by parts to $$\f{\psi(\hat F(z)) }z d \tilde F_s(z).$$ 
The requirement at \eqref{eq:inequality} arises from the extra term $\psi(\hat F(z))$. 
Also, note that we work in the norm $\dnorm{\cdot}$ to obtain the constant $ (2-\delta)$ in \eqref{eq:inequality}.

\begin{proof}[Proof of \thref{prop:cond1}]

Let $\mathbf F \in \mathfrak F^\alpha$. We consider the rescaled processes $\tilde F_t(x) = F( e^{t} x)$, $\tilde F^{\Psi}_t(x) = \hat F(e^t x)$. It then holds that $\tilde{\mathbf F} = \tilde{ \mathscr C} (\tilde {\mathbf F})$ where  
$$\tilde {\mathscr C}(\tilde{\mathbf F})_t(x) =  \tilde F_0(x) + \int_0^t e^s x^2 \int_x^\infty \f{ \psi(\hat F(z) )}{z} d \tilde F_s(z) ds.$$

Our goal is to prove the distance between $\tilde{ \mathbf F }$ and $\alpha \tilde {\mathbf{\hat F} }^*$ is decreasing in $t$:
\begin{align}
\partial_t \dnorm{ \tilde F_t - \tF } = \int_0^\infty x ^{-1 - \delta} \partial _t | \tilde F_t(x) - \tF(x) | dx \leq 0 \label{eq:suffices1}.
\end{align}

We start by differentiating under the integral sign 
\begin{align*}
\partial_t \tilde \CC( \tilde{\mathbf F } )_t(x) = e^{t} x^2 \int_x^\infty \f{ \psi(\hat F(z))} {z} d \tilde F_t(z)
\end{align*}
to write for each $x\geq0 $ the dynamics for the difference $\tilde F_t(x) - \tF(x)$ as
$$\partial _t ( \tilde F_t(x) - \tF(x)) = e^{t} x^2 I_t(x),$$
$$I_t(x) = \int_x^\infty \f{ \psi(\hat F(z) )}{z} \partial_z( \tilde F_t(z) - \tF(z) ) dz.$$
Multiply both sides by $\sgn( \tilde F_t - \tF)$ to obtain
\begin{align*}
e^{-t} \partial _t | \tilde F_t(x) - \tF(x) | = x^2 \begin{cases} \sgn ( \tilde F_t(x) - \tF(x)) I_t(x), & \tilde F_t(x) \neq \tF(x) \\ 0, & \tilde F_t(x) = \tF(x) \end{cases}.
\end{align*}
Let $\hat f(z) =  z\psi'(\hat F(z)) \hat F'(z) - \psi(\hat F(z)) $. An application of integration by parts to the integral gives
$$I_t(x) = - \f{\psi(\hat F(x)) }{x} (\tilde F_t(x) - \tF(x) ) + \int_x^\infty \f{ \hat f(z)  }{z^2} (\tilde F_t(z) - \tF(z) ) dz.$$ The previous two equations therefore yield
\begin{align*}
e^{-t} \partial _t | \tilde F_t(x) - \tF(x) |
& \leq - x \psi(\hat F(x)) | \tilde F_t(x) - \tF(x)| + x^2  \int_x^\infty |\hat f(z) | \f{ |\tilde F_t(z) - \tF(z)| }{z^2} dz.
\end{align*}
We next multiply both sides by $x^{-1 - \delta}$ and integrate with respect to $x$ from $0$ to infinity to obtain the bound 
\begin{align}
e^{-t} \int_0^\infty x ^{-1-\delta} \partial_t | \tilde F_t(x) - \tF(x) | dx  &\leq \int_0^\infty   - \psi(\hat F(x)) \f{| \tilde F_t(x) - \tF(x)|}{x^{\delta}} dx \nonumber \\
& \qquad \qquad 
 + \int_0^\infty x^{1- \delta} \int_x^\infty |\hat f(z) | \f{ | \tilde F_t(z) - \tF(z)|}{z^2} dz dx . \nonumber
\end{align}
An application of Fubini's theorem lets us rewrite the second integral as 
\begin{align*}
\int_0^\infty x^{1- \delta} \int_x^\infty |\hat f(z) | \f{ | \tilde F_t(z) - \tF(z)|}{z^2} dz dx &= \int_0^\infty |\hat f(z) | \f{ | \tilde F_t(z) - \tF(z)|}{z^2} \int_0^z x^{1- \delta } dx dz \\
&=\int_0^\infty (2 - \delta)\inv |\hat f(z) | \f{| \tilde F_t(z) - \tF(z) | }{z^{\delta  }} dz.
\end{align*}
Hence we can combine the integrals to obtain the bound
\begin{align}
e^{-t} \int_0^\infty x ^{-2} \partial_t | \tilde F_t(x) - \tF(x) | dx & \leq \int_0^\infty   \Big((2-\delta)\inv |\hat f (z)| - \psi(\hat F(z))\Big) \f{|\tilde F_t(z) - \tF(z)|}{z^{\delta}}  dz . \nonumber 
\end{align}
Our hypothesis \eqref{eq:inequality} guarantees that the term inside the integral:
$$(2-\delta)\inv |\hat f (z)| - \psi(\hat F(z))\leq 0.$$
Therefore \eqref{eq:suffices1} holds. 
This establishes that 
\begin{align}\dnorm{\tilde F_t - \tF} \leq \dnorm{ \tilde F_0 - \alpha \tilde F^\Psi_0} = \dnorm{ F_0 - \alpha \hat F}. \label{eq:=1}
\end{align}
 A change of variables $x = e^{-t} z$ gives 
\begin{align}
\dnorm{F_t - \alpha \hat F}  &= \int_0^\infty x^{-1 - \delta} |  F_t(x) - \alpha \hat F(x)|dx\nonumber \\
 &= e^{-\delta t} \int_0^\infty z^{-1 - \delta} | \tilde F_t(z) -  \tF(z) | dz \nonumber\\
 &= e^{-\delta t}\dnorm{ \tilde F_t - \tF}\nonumber\\
 &\leq e^{-\delta t} \dnorm{ F_0 - \alpha \hat F},
\end{align}
where at the last line we apply \eqref{eq:=1}.

It remains to prove that $\dnorm{ F_0 - \alpha \hat F} \leq C$, for some $C>0$. By assumption, $\mathbf F\in \mathcal X_1$ and therefore $\xnorm{ F_0} \leq 1$. As $0 \leq F_0(x) \leq 1$ we can break up the integral and use integrability of $x^{-1- \delta}\ind{x >1}$:
$$\int_0^\infty x^{-1-\delta} F_0(x) dx \leq \int_0^1 x^{-2} F_0(x) dx + \int_1^\infty x^{-1-\delta} dx \leq \xnorm{F_0} + \delta^{-1} \leq 1 + \delta \inv.$$
Similarly, $\dnorm{\alpha \hat F} \leq 1  + \delta \inv$. Apply the triangle inequality to conclude $\dnorm{F_0 - \alpha \hat F} \leq \dnorm{F_0} + \dnorm{\alpha \hat F} \leq 2(1+ \delta\inv).$
\end{proof}

\subsection{\thref{master}} \label{proofs}

In \thref{master} we prove that $ A_t^\alpha$ and $A_t$ have similar properties. Each statement requires some manipulation. Fortunately \cite{elliot} contains much of the heavy-lifting. 
We make one remark concerning the proof of \ref{continuous}. In \cite{elliot} they prove continuity of an operator $\mathscr S^\Psi$ with domain $\mathcal X$. Our operator $\mathscr C$ has domain $\mathcal X \times \mathcal X$. This makes the proof more involved, and also restricts us to proving continuity in sequences of the form $ (\mathbf F^{(n)}, \mathbf  A^{(n)} )$.

\begin{proof}[Proof of \ref{rate}]  
   
      The equality $\xnorm{\alpha \hat F} = \alpha$ is \cite[Lemma 3.5]{elliot}. For the other equality, take $I^{\alpha, (t)}_j$ to be the length of an interval in $[0,\alpha]$. Define the measure $\mu_t^\alpha = e^{-t} \tsum_{1}^{N_t^\alpha} \delta_{e^tI^{\alpha, (t)}_j}.$
    This gives $\mu_t^\alpha$ is the empirical distribution of rescaled interval lengths. We can then write $$A_t^\alpha(x) = \int_0^xy \mu_t(dy).$$ Applying Fubini's theorem shows that $$\xnorm{ A_t^\alpha} = \int_0^\infty x^{-2} \int_0^xy \mu_t^\alpha(dy) dx = \int_0^\infty \mu_t^\alpha(dy) = e^{-t}N_t^\alpha.$$
\end{proof}

\begin{proof}[Proof of \ref{tight}]
Recall that a family of distributions $(F_t)_{t \geq 0}$ is \emph{tight} if for all $\epsilon>0$ there exists $N$ such that $F_t(N) \geq 1- \epsilon$ for all $t \geq 0$.  
\cite[Proposition 6.3]{elliot} implies $(A_t)_{t \geq 0}$ is tight. Fix $\epsilon>0$ and let $N$ be such that $A_t(N) \geq 1 - \alpha \epsilon$ for all $t \geq 0$.
The relationship at \eqref{eq:key} ensures $A^{\alpha}_t(N) + A^{\alpha_+}_t(N) \geq 1 - \al \epsilon.$
As $A_t^{\ep} \leq \al $ and $A_t^{\alpha_+} \leq 1 -\al $, this inequality could only hold if $A_t^\ep(N) \geq \alpha - \al \epsilon$ for all $t\geq 0$. Hence, $(\f 1 \al A_t^\al)_{t \geq 0 }$ is tight.
\end{proof}

\begin{proof}[Proof of \ref{equi}]
Recall, that a family of functions $( \mathbf F^{(n)})_{n \in \mathbb N}$ in $\mathcal X$ is asymptotically equicontinuous if for every compact $K \subset [0,\infty)$,
$$\lim_{\delta \to 0} \lim_{n \to \infty }  \sup_{\substack{s,t \geq 0 \\ |s-t| \leq \delta } }  \int_K  | F_s^{(n)}(x) - F_s^{(n)}(x) | dx  = 0.$$ 
The proof is similar to \cite[Lemma 7.5]{elliot}. The idea is that it suffices to show the existence of a $\delta_0>0$ and constant $C$ so that for every $0<\delta_1 < \delta_0$ there exists almost surely a $T_{\delta_1}< \infty$ so that 
\begin{align}\sup_{ t \geq T_{\delta_1}, 0 \leq \delta \leq \delta_1} \int_0^\infty \f{ | A_{t+ \delta}^\alpha(x) - A_t^\alpha(x) | }{ x^2} dx \leq C \delta_1.\label{eq:as2}	
\end{align}

This is sufficient since we for any $\delta_1 >0$ and any $M>0$, almost surely
\begin{align*}
\lim_{n \to \infty} \sup_{ s,t \geq 0, |s-t| \leq \delta_1} \int_0^M | A_s^{\alpha, (n)}(x) - A_t^{\alpha, (n)}(x) | dx & \leq \sup_{ t \geq T_{\delta_1}, 0 \leq \delta \leq \delta_1} \int_0^\infty \f{ | A_{t+ \delta}^\alpha(x) - A_t^\alpha(x) | }{ x^2} dx \\
	&	\leq M^2 C \delta_1.
\end{align*}
As this holds jointly with probability 1 for a countable sequence of $\delta_1$ going to 0 and $M \in \mathbb N$, the asymptotic equicontinuity of $(A^{(n)})_{n \geq 0}$ follows. 

The formula at \eqref{eq:as2} follows from the fact that
 $\tilde A^\alpha_t$ satisfies the monotonicity condition, for any $\delta >0$, 
	\begin{align}
	\tilde A_t^\alpha(x)	\leq \tilde A^\alpha_{t+\delta}( e^{-\delta} x ) \leq \tilde A^\alpha_{t+ \delta}(x). \label{eq:ae1}
	\end{align}
 Another necessary fact is that number of points kept in $[0,\alpha]$ from time $t$ to $t+\delta$ is bounded by the number of points added to $[0,1]$ in that same time interval. Formally, for any $\delta >0$ we have $N_{t+\delta}^\alpha - N_t^\alpha \leq N^1_{t+\delta} - N^1_t$. This lets us deduce the equivalent for $N_t^\alpha$ as for $N_t$ in \cite[Lemma 7.6]{elliot}. Namely, that there is a $\delta >0$ so that for every $0 < \delta < \delta_0$ there exists almost surely a $T_\delta < \infty$ so that
 $$\sup_{t \geq T_\delta} N_{t+\delta}^\alpha - N_t^\alpha \leq 2 \delta e^t.$$
 The argument finishes by using the formula from \thref{master} \ref{rate} for $N_t^\alpha$ in terms of $\xnorm{ A_t^\alpha}$. See the proof of \cite[Lemma 7.5]{elliot} for further details.
\end{proof}


\begin{proof}[Proof of \ref{noise}]
The proof is similar to the decay of the noise subsection in \cite[Section 7]{elliot}. The idea is to bound the martingale $\mathbf M^\alpha$ by computing various moments of the underlying process $\mathbf B^\alpha$.   We can use the same bounds as in \cite{elliot} because points are added to $[0,\alpha]$ no faster than to $[0,1]$. This ensures that $B^\ep(s,u,v,x) \leq B(s,u,v,x)$. Here $B(s,u,v,x)$ is the function defined at \cite[(3)]{elliot}. 
\end{proof}

\begin{proof}[Proof of \ref{continuous}]
Suppose that $\mathbf F^{(n)} \xto \mathbf F$. An equivalent notion of convergence in the topology of local uniform convergence is that $\mathbf F^{(n)} \overset{\mathcal X} \to \mathbf F$ if and only if for all compact $K \subset [0,\infty)$
$$\lim_{n \to \infty} \sup_{0 \leq s \leq t} \int_K |F^{(n)}_s(x) - F_s(x) | dx = 0.$$

\cite[Theorem 7.1]{elliot} implies $\mathbf A^{(n)} \xto \mathbf F^*$. Thus it suffices to prove for any fixed $T > 0$ and $K >0$ 
\begin{align}
\int_0^K |\mathscr C(\mathbf F, \mathbf F^*)_t(x)- \mathscr C(\mathbf F^{(n)}, \mathbf A^{(n)})_t(x)  | dx \to 0 \label{eq:0.1}
\end{align}
uniformly for $t \leq T$.
For fixed $n$ we can write
\begin{align*}
\mathscr C(\mathbf F^{(n)}, \mathbf A^{(n)})_t(x)&= F_0^{(n)}(x) + \int_0^t (e^{s-t} x)^2 \int_{e^{s-t} x}^\infty \f{ \psi(A_s^{(n)}(z))}{z} d F^{(n)}_s(z) ds.
\end{align*}
If we write $\psi( A_s^{(n)}(z)) = \psi(\hat F(z)) + \psi( A_s^{(n)}(z)) - \psi(\hat F(z))$  the above becomes
\begin{align*}
\mathscr C(\mathbf F^{(n)}, \mathbf A^{(n)})_t(x)&= {\mathscr C( \mathbf F^{(n)}, \mathbf F^*)_t(x)} + {\int_0^t (e^{s-t} x)^2 \int_{e^{s-t} x}^\infty \f{ \psi( A_s^{(n)}(z)) - \psi(\hat F(z)) }{z} d F^{(n)}_s(z) ds}.
\end{align*}
We can then bound the left side of \eqref{eq:0.1} by
\begin{align}
\int_0^K |\mathscr C(\mathbf F &, \mathbf F^*)_t(x) - \mathscr C( \mathbf F^{(n)}, \mathbf F^*)_t(x) | dx \label{eq:term1} \\
&+ \int_0^K {\int_0^t (e^{s-t} x)^2 \int_{e^{s-t} x}^\infty \f{ |\psi( A_s^{(n)}(z)) - \psi(\hat F(z)) |}{z} d F^{(n)}_s(z) ds} dx. \label{eq:term2}
\end{align}

It suffices to show that as $n \to \infty$ each summand converges to zero uniformly for $t \leq T$.
    \vspace{.2 cm}
    
\subsubsection*{First summand}  Start by bounding the summand at \eqref{eq:term1} by
    \begin{align*}
   \int_0^K |F_0(e^{-t}x) - F^{(n)}_0(e^{-t}x)| dx + \int_0^K\int_0^t (e^{s-t} x)^2 \bigg     |\int_{e^{s-t} x}^\infty \f{ \psi(\hat F(z) )}{z} d( F_s(z) - F^{(n)}_s(z) )  \bigg|ds dx.
    \end{align*}
The first quantity goes to zero uniformly for $t \leq T$ by the definition of $\mathbf F^{(n)} \xto \mathbf F$ since a change of variables gives
$$\int_0^K |F_0(e^{-t}x) - F^{(n)}_0(e^{-t}x)| dx \leq e^t \int_0^K | F_0(x) - F^{(n)}_0(x) | dx.$$

  Expand the interior of the second quantity with integration by parts and take the absolute value signs inside to bound it by
$$  \underbrace{\f{ \psi(\hat F(e^{s-t} x))}{e^{s-t}x } | F_s(e^{s-t} x) dx - F_s^{(n)}(e^{s-t} x)|}_{\text{term one}} + \underbrace{\int_{e^{s-t} x}^\infty  \bigg| \f{d}{dz}\f{ \psi(\hat F(z) )}{z} \bigg| | F_s(z) - F^{(n)}_s(z) | dzdx}_{\text{term two}} .$$
Multiply term one by $(e^{s-t} x)^2$ and integrate so it becomes
$$\int_0^K \int_0^t (e^{s-t} x) \psi( \hat F(e^{s-t}x) ) |F_s(e^{s-t} x)- F_s^{(n)}(e^{s-t} x)| dsdx.$$
Since $\hat F$ is a distribution function and $\psi$ is continuous we have $(\psi \circ \hat F)(u) \leq \sup_{u \in [0,1]} \hat \psi(u) < D < \infty$ for some constant $D$. Thus, the above is bounded by
$$D\int_0^K \int_0^t (e^{s-t} x)|F_s(e^{s-t} x)- F_s^{(n)}(e^{s-t} x)| dx.$$


The above goes to zero by the definition of $\mathbf F^{(n) } \overset{\chi} \to \mathbf F$. 
As for term two, we differentiate to rewrite it as
\begin{align}
\int_{e^{s-t} x}^\infty\f{| z\psi'(\hat F(z) ) \hat F'(z) - \psi(\hat F(z))| }{z^2}  |F_s(z) - F^{(n)}_s(z) | dz . \label{eq:1.1}
\end{align}
Our additional hypothesis is that $z \hat F'(z)$ is bounded. Since the range of $\hat F$ is contained in the compact interval $[0,1]$ and $\Psi \in C^2$ we have $\psi\circ \hat F$ and $\psi'\circ \hat F$ are also bounded. Therefore, $C = \sup_{0 \leq z \leq \infty} | z \hat F'(z)\psi'(\hat F(z) ) - \psi( \hat F(z) )|< \infty$. It follows that \eqref{eq:1.1} is less than
\begin{align}
C\int_{e^{s-t} x}^\infty\f{1}{z^2}  |F_s(z) - F^{(n)}_s(z) | dz  \label{eq:1.2}.
\end{align}
Finally we are in the position of $I_2$ from \cite[Lemma 3.3]{elliot} and can conclude that \eqref{eq:1.2} goes to zero uniformly for $t \leq T$.
\item

    \vspace{.2 cm}
\subsubsection*{Second summand} Fix $M>0$ and for any function $f:[0,\infty) \to [0,1]$ define $ f^M = f|_{[0,M]}$ to be the restriction to the domain $[0,M]$. We have in \cite[Theorem 7.1]{elliot} that $A^{M}$ converges pointwise to $\hat F^M$. Observe that each $A^{M}_t$ is an increasing function with compact domain, and $\hat F^M$ is continuous by \cite[Lemma 3.5]{elliot}. Together these imply (see  \cite[exercise 7.13]{rudin}) that for any $\epsilon >0$ there exists $t_\epsilon$ such that for all $z \in [0,M]$
$$\sup_{t \geq  t_\epsilon} | A_t^{M}(z) - \hat F_t^M(z)| < \epsilon.$$
 Because the functions $A^{(n)}_t$ are translates of $A_t$ it follows that for all $n > t_\epsilon$ we have
\begin{align*} \sup_{t \geq 0 }| A_t^{(n),M}(z) - \hat F_t^M(z)|&\leq \sup_{t \geq t_\epsilon} | A_{t}^M(z) - \hat F_t^M(z)| < \epsilon .\end{align*}
As the functions $A_t^{(n)}$ and $\hat F$ are supported on $[0,1]$, we have their compositions with $\psi$ are uniformly continuous. We conclude that there exists $n_0$ such that for all $z \in [0,M]$
\begin{align}
\sup_{t \geq 0 } |\psi( A_t^{(n)}(z)) - \psi(\hat F(z))| < \epsilon, \qquad  \text{for } n \geq n_0. \label{eq:2.1}
\end{align}

We truncate the integral then apply \eqref{eq:2.1} to bound the absolute value of \eqref{eq:term2} by
\begin{align}
\epsilon \int_0^K \int_0^t (e^{s-t} x)^2 &\int_{e^{s-t} x}^M \f{ 1 }{z} d F^{(n)}_s(z) ds dx \label{eq:2.00}\\
&+ \int_0^K\int_0^t (e^{s-t} x)^2 \int_{M}^\infty \f{ |\psi( A_s^{(n)}(z)) - \psi(\hat F(z))| }{z} d F^{(n)}_s(z) ds dx \label{eq:2.01}.
\end{align}

We can use the fact that $F_s^{(n)}(z) \leq 1$ and bound the inside integral of \eqref{eq:2.00} by
$$\f{ 1 }{ e^{s-t x}} \int _{e^{s-t}x}^M d F_s^{(n)}(z) \leq \f{2}{ e^{s-t}x}.$$
Thus \eqref{eq:2.00} is bounded by
$$\epsilon \int_0^K \int_0^t 2 e^{s-t} x ds dx\leq \epsilon (1 - e^{-t}) K^2 \leq \epsilon K^2.$$
As $K$ is fixed, this can be made arbitrarily small.

Lastly we consider \eqref{eq:2.01}. Since $\sup_{u \geq 0} \psi(u) = D < \infty$ we use similar estimates as in \eqref{eq:2.00} and start with the bound
\begin{align*}
\int_0^K\int_0^t (e^{s-t} x)^2 \int_{M}^\infty &\f{ |\psi( A_s^{(n)}(z)) - \psi(\hat F(z))| }{z} d F^{(n)}_s(z) ds dx \\
&\qquad \qquad \leq 4 D \int_0^K\int_0^t (e^{s-t} x)^2 \f{ 1 }{ M} ds dx \\
& \qquad \qquad  \leq \f {4D K^3(1 - e^{-2t}) } {6M} .
\end{align*}
Since $M$ can be made arbitrarily large, this can be made as small as we like. Therefore, the absolute value of \eqref{eq:term2} can be bounded by any $\epsilon>0$ uniformly for $t \leq T$. 
\end{proof}

\begin{lemma} \thlabel{lem:extra}
If $\Psi$ satisfies $(\text{\emph{C}}^2)$ and either $\psi(1) >0$ or $\Psi(u) = 1 - (1-u)^k$ for some positive integer $k$ then $\sup_{z \geq 0 }z \hat F'(z) < \infty.$ 
\end{lemma}

\begin{proof}
\cite[Proposition 8.2]{elliot} states that when $\psi(1)>0$ it holds that $\hat F'(x) \leq C e^{-ax}$ for some constants $C,a >0$. Additionally, for the $\min$-$k$ process $(\Psi(u) = 1 - (1-u)^k)$ it is shown in \cite[Proposition 8.4]{elliot} that $\hat F'(x) \leq C_k x^{ -1 - \epsilon_k}$ for some $C_k,\epsilon_k >0$. Note that $\sup _{k \geq 0} C_k < \infty$ and $\epsilon_k \to 0$.  
\end{proof}

\begin{cor} \thlabel{cor:extra}
	
From \thref{lem:extra} $z \hat F'(z)$ is bounded for all interpolations of the $\max$-$k$ and $\min$-$k$ processes.
\end{cor}

We remark that it appears boundedness of $z \hat F'(z)$ does not necessarily hold for general $\Psi$. At the very least it does not obviously follow from \eqref{eq:integro} or \eqref{eq:diff}.



\section{Proving \thref{cor:main}} \label{sec:inequalities}

For this entire section we will let $F$ denote $F^{\Psi}$. To establish \eqref{eq:inequality}, we rely almost entirely on \eqref{eq:integro} and \eqref{eq:diff}. For convenience we rerecord them here:
\begin{align} |z\psi'(F(z))  F'(z) - \psi( F(z))|  \leq (2 - \delta) \psi( F(z) ) \tag{1}, \qquad \delta \in (0,1],
\end{align}
\begin{align}
F'(z)  = z \int_z^\infty \f 1 y d \Psi( F(y)), \tag{2}
\end{align}
\begin{align}
z  F''(z) -  F'(z)  + z \psi( F(z))  F'(z) = 0. \tag{3}
\end{align}
	
We start with the proof of \thref{cor:main}. It follows from a sequence of lemmas.

\begin{proof}[Proof of \thref{cor:main}]
    First off we need the conclusion of \thref{cor:extra} to guarantee \thref{master} \ref{continuous} holds for the interpolations we consider.
         Equidistribution for the $\max$-2 process then follows from \thref{lem:special} by taking $p_2=1$. The fact that the interpolation that is $60$\%-$\min$-2 satisfies \eqref{eq:inequality} follows by taking $p_{-2} = .6$ in \thref{lem:pinch}. Part three (for general interpolations) follows from \thref{lem:pinch2}. 
\end{proof}

 Now we give the proofs of the necessary lemmas. We break this up into two sections: one for interpolations of $\max$-2 and $\min$-2 processes and the other for general interpolations. 
 
\subsection{Interpolations of $\min$-$2$ and $\max$-$2$}

Fix  $p_{-2}, p_2 \in [0,1]$ with $p_2+p_{-2} = 1$. We will work exclusively in this subsection with $\Psi$ that are interpolations of the $\min$-2 and $\max$-2 process. Thus,
\begin{align*}
\Psi(u) &= p_2 u^2 + p_{-2} (1 - (1-u)^2), \\
\psi(u) &= 2 p_2 u + 2 p_{-2} (1- u), \\
\psi'(u) &= 2 p_2 - 2 p_{-2},
\end{align*}
 This is the distribution function (and derivatives) for an interpolation where at each step we add a point according the $\min$-2 process with probability $p_{-2}$ and according to the $\max$-2 process with probability $p_2$.

Our first lemma establishes \eqref{eq:inequality} holds so long as $p_{-2} \leq p_2$. Note that the case $p_2=1$ is the $\max$-2 process.

\begin{lemma}\thlabel{lem:special}
   If $p_{-2} \leq p_2$ then \eqref{eq:inequality} holds.

\end{lemma}

\begin{proof}
Dropping the constant $2-\delta$ from the right side of \eqref{eq:inequality} it suffices to prove that
\begin{align}
| \psi(F(z)) - z\psi'(F(z)) F'(z)  |  \leq \psi(F(z)). \label{eq:sufficient2}
\end{align}
We break into two cases:
\begin{itemize}
    
    \item First suppose $\psi(F(z)) \geq z \psi'(F(z))F'(z)$ so that \eqref{eq:sufficient2} reduces to proving that $$- z \psi'(F(z)) F'(z) \leq 0.$$ As $F$ is increasing we know $F'(z) \geq 0$. The hypothesis $p_{-2} \leq p_2$ guarantees that $\psi'(F(z)) \geq 0$. Thus, the inequality is satisfied.
    \item Next, suppose $ \psi( F(z) ) \leq z\psi'(F(z)) F'(z)$. Rearranging \eqref{eq:sufficient2} 
     we seek to show
    $$2(p_2 - p_{-2})  zF'(z) \leq 2 \psi( F(z) ).$$
     Note that both sides are zero at $z=0$. By the fundamental theorem of calculus it then suffices to prove the above inequality holds for the derivatives.
       Differentiating and again using the fact that $\psi'(F(z)) = 2 (p_2 - p_{-2})$ reduces the problem to establishing
     \begin{align*}2(p_2 - p_{-2}) ( z  F''(z) + F'(z) ) \leq 4(p_2 - p_{-2}) F'(z).\end{align*}
      After some algebra this is equivalent to
      \begin{align}
          z F''(z) \leq F'(z). \label{eq:reduction2}
      \end{align}
     From \eqref{eq:diff} we know that $z F''(z) = F'(z)-  z \psi( F(z) )  F'(z)$. Substitute this 
 into \eqref{eq:reduction2} and we have a sufficient condition 
 is that
     $$F'(z)- 2 z \psi(F(z)) F'(z) \leq  F'(z).$$
    This holds as $F'(z)$ and $\psi( F(z))$ are nonnegative.
\end{itemize}
\end{proof}

To prove \eqref{eq:inequality} holds when $p_{-2} > p_2$ requires a different analysis of the differential equation at \eqref{eq:diff}. 
\noindent \thref{lem:zF'} shows $zF'(z)$ can be bounded in terms of $p_{2}$. 

\begin{lemma} \thlabel{lem:F'1}
If $p_{-2} > p_2$ then 
$\lim_{\epsilon \to 0 } {F'(\epsilon)}/{\epsilon} \leq 2.$
\end{lemma}

\begin{proof}
Starting from the formula at \eqref{eq:integro} then integrating by parts gives
\begin{align}
\lim_{\epsilon \to 0} \f{F'(\epsilon)}{\epsilon} &= \int_0^\infty \f{ 1}{y} d\Psi(F(y))
=\xnorm{\Psi \circ F}. \label{eq:F'1}
\end{align}
Plugging into $\Psi$ we have
\begin{align}
\Psi(F(y)) &= p_2 F(y)^2 + p_{-2}(1 - (1-F(y))^2)\\
& =F(y)[ (p_2- p_{-2}) F(y) + 2p_{-2} ] \nonumber.
\end{align}
The hypothesis $p_2 < p_{-2}$ means an upper bound for the above is
\begin{align}
\Psi(F(y)) &\leq 2 p_{-2} F(y) \leq 2 F(y). \label{eq:F'11}  
\end{align}
\thref{master} \ref{rate} implies that $\xnorm{F} = 1$. It follows from \eqref{eq:F'1} and \eqref{eq:F'11} that 
\begin{align*}
\lim_{\epsilon\to 0} \f { F'(\epsilon)}{\epsilon} \leq 2\xnorm{F} = 2.
\end{align*}

\end{proof}

\begin{lemma} \thlabel{lem:zF'}
It $p_{-2} > p_2$ then 
$$z F'(z) \leq  2( p_2 e)^{-2 }.$$
\end{lemma}

\begin{proof}

Integrate \eqref{eq:diff} as in \cite[Proposition 8.1]{elliot} so that for any $\epsilon >0$
\begin{align*}F'(z) = \f{F'(\epsilon)}{\epsilon}z \exp\left( - \int_\epsilon^z \psi(F(y)) dy \right).\end{align*}
Taking $\epsilon \to 0$ and applying \thref{lem:F'1} gives
\begin{align}F'(z) \leq 2z \exp\left( - \int_0^z \psi(F(y)) dy \right).\label{eq:F'}\end{align}
We observe that $\psi(F(y)) = 2p_{2}F(y) + 2p_{-2}(1 - F(y)).$ Since we are assuming $p_{-2} > p_2$ and know that $F(y) \leq 1$ we obtain a lower bound by evaluating at $\psi(1)$:
\begin{align}\psi(F(y)) \geq \psi(1) = 2 p_2. \label{eq:psi}
\end{align}
 Applying this to \eqref{eq:F'} and multiplying by $z$ gives
\begin{align*}
zF'(z) &\leq 2 z^2 e^{- 2p_2 z}.
\end{align*}
The maximum of $z^2 e^{-2p_2 z}$ is at $z= 1/p_2$. Plug this in above to obtain the claimed bound.
\end{proof}

\begin{lemma} \thlabel{lem:pinch}
If $p_2 < p_{-2} \leq .6$ then \eqref{eq:inequality} holds.
\end{lemma}

\begin{proof}
Using the triangle inequality on the left side of \eqref{eq:inequality} it suffices to find $\delta$ such that for all $z \geq 0$
\begin{align}
| z \psi'(F(z)) F'(z)| + |\psi(F(z)) | \leq (2-\delta)\psi(F(z)) . \label{eq:first}
\end{align}
Because $F$ is a distribution function, we know that $F'\geq 0$. Also, note that $$(2-\delta)\psi(F(z)) - |\psi(F(z))| \leq (1- \delta) \psi(F(z)).$$ Thus, to establish \eqref{eq:first} it is enough to prove
\begin{align} z F'(z) \leq \f {(1- \delta)\psi(F(z) )} {|\psi'(F(z))|}, \qquad  \text{for $z \geq 0$}. \nonumber 
 \end{align}
We have from \eqref{eq:psi} that $\psi(u) \geq  2p_2$ and can compute $|\psi'(u)| = 2 |p_2 - p_{-2}|$. 
It then suffices to prove
\begin{align*}
z F'(z) \leq \f{ p_2(1- \delta)}{|p_2- p_{-2}|}.
\end{align*}
By \thref{lem:zF'} it suffices to choose $\delta$, $p_{-2}$ and $p_{2}$ so that
$$ 2\left( p_2 e \right)^{-2} \leq \f{ p_2(1- \delta)}{|p_2- p_{-2}|}.$$
Combining with our hypotheses we have the following system of constraints
\begin{align}
2e^{-2}|p_2 - p_{-2}| &\leq (1-\delta)(p_2)^3, \nonumber\\
p_2 + p_{-2} &= 1 ,\nonumber\\
p_2 &< p_{-2},\nonumber \\
0 &< \delta \leq 1. \nonumber
\end{align}
Take $\delta \to 0$ and use the fact that $p_{-2}$ is assumed to be larger than $p_{2}$, and the solution must be strictly smaller than the real root of the cubic
$$\f 2 { e^2} (p_{-2} - (1- p_{-2})) = (1 - p_{-2})^3.$$
This is approximately $.61$, thus $p_{-2} \leq .6$ lies in the solution set.
\end{proof}

\begin{remark}
The bound $p_{-2} \leq .6$ could be optimized further in the preceding lemmas, but the gain would be marginal. Something like $p_{-2} \leq .68$ is the best that comes out of optimizing our argument. We sacrifice this marginal gain for the sake of clarity. 
\end{remark}



\subsection{General interpolations of $\max$-$k$, uniform and $\min$-$k$ processes}

We will reprove versions of the previous three lemmas for more general interpolations. Let $\mathbf p =(p_k)_{k \neq -1,0}$ be a probability measure on $\mathbf Z \setminus \{-1,0\}$. In this subsection we consider the interpolations
\begin{align*}
\Psi(u) &= p_1 u + \tsum_{k \geq 2} p_k u^k + p_{-k} (1 - (1-u)^k).
\end{align*}
Define $C_{\mathbf{p}} = \sum_{k \geq 2} k (k-1)(p_k + p_{-k}).$ This constant arises because $\sup_{u \geq 0} |\psi'(u)| \leq C_{\mathbf{p}}.$ First we give a bound on $F'$ that holds for any $\Psi$-process.

\begin{lemma} \thlabel{lem:F'2}
Let $\Psi$ satisfy $(C)$ and $(D)$. For all $z \geq 0$ it holds that ${F'(z)} \leq 1.$
\end{lemma}

\begin{proof}
This follows from  a simple bound on \eqref{eq:integro}:
\begin{align*}
F'(z) = z \int_z^\infty \f{ \psi(F(y)) }{y}F'(y) dy 
&\leq  z \cdot \f 1 z \int_z^\infty \psi(F(y))F'(y) dy \\
&= \Psi(1) - \Psi(F(z)).
\end{align*}
Since $\Psi(1) =1$ we conclude that $F'(z) \leq 1.$ 
 
\end{proof}

Now let us return to the setting where $\Psi$ is an interpolation of $\max$-$k$, uniform and $\min$-$k$ processes given by $\mathbf p$.

\begin{lemma} \thlabel{lem:zF'2}
Suppose that $p_{1} > 0$. It holds that 
$$z F'(z) \leq \f{2e^{-1}}{(p_1)^2} .$$
\end{lemma}

\begin{proof}

Integrate \eqref{eq:diff} as in \cite[Proposition 8.1]{elliot} so that for any $\epsilon >0$
\begin{align*}F'(z) = \f{F'(\epsilon)}{\epsilon}z \exp\left( - \int_\epsilon^z \psi(F(y)) dy \right).\end{align*}
Taking $\epsilon = 1$ and applying \thref{lem:F'2} gives
\begin{align}F'(z) \leq z \exp\left( - \int_1^z \psi(F(y)) dy \right).\label{eq:F'2}\end{align}
Notice that 
\begin{align}\psi(u)  = p_1 + \tsum_{k \geq 2} k [p_k  u^{k-1} + p_{-k}  (1- u)^{k-1}] \geq p_1. \label{eq:psi2}
\end{align}
Apply this to \eqref{eq:F'2} then multiply by $z$ to obtain the bound
\begin{align*}
zF'(z) &\leq   e^{p_1} z^2 e^{-p_1z} 
\end{align*}
The maximum of $z^2 e^{-p_1 z}$ is at $z= 2/ p_1$. Plug this in above to obtain the claimed bound.
\end{proof}

\begin{lemma} \thlabel{lem:pinch2}
If $C_{\mathbf{p}} \leq \f 12$ then \eqref{eq:inequality} holds.
\end{lemma}

\begin{proof}
As in \thref{lem:pinch} it suffices to show
for some $\delta \in (0,1]$ and all $z \geq 0$
\begin{align} z F'(z) \leq \f {(1- \delta)\psi(F(z) )} {|\psi'(F(z))|} . \nonumber 
 \end{align}
We have from \eqref{eq:psi2} that $\psi(u) \geq p_1$ and can compute $$|\psi'(u)| \leq \sum_{k \geq 2} k(k-1)| u^{k-2} - (1-u)^{k-2}| \leq C_{\mathbf{p}}.$$ 
It then suffices to prove
\begin{align}
z F'(z) \leq \f{(1- \delta) p_1}{C_{\mathbf{p}}}. \label{eq:end}
\end{align}
By \thref{lem:zF'2} and the hypothesis $C_{\mathbf{p}} \leq 1/2$ it suffices to choose the $p_k$ so that
$$ \f{2e^{-1}}{(p_1)^2}  \leq  2(1- \delta) p_1.$$
Rewriting and letting $\delta \to 0$ we require that $ e^{-1/3} <p_1$. It is easy to verify (by just checking the case $p_k =0$ for $k \neq 1, 2$) that we must have $\sum_{k \neq 1} p_k<1/4$ in order to satisfy $C_{\mathbf p} < 1/2$. Thus, $p_1 > 3/4$. Since $e^{-1/3} \approx .71 < 3/4 =p_1$ the above displayed inequality holds.
\end{proof}

%

\subsubsection*{Acknowledgments}

Much thanks to Elliot Paquette and Pascal Maillard for many useful conversations. The first referee's careful reading and suggestion to generalize to arbitrary $\Psi$ processes are greatly appreciated. Itai Benjamini is the source of a very similar model that spurred this research. Toby Johnson and Bal\'azs Gerencs\'er provided nice suggestions on earlier drafts. I am grateful to my advisor Christopher Hoffman for encouraging me to stick with the problem and his advice to consider a small subinterval. Gerandy Brita Montes de Oca's assistance with reading and understanding \cite{elliot} was very helpful. 
Tatiana Toro and Shirshendu Ganguly gave some useful advice about the operator $\mathscr C$. Thanks to Chloe Huber and Chris Fowler for being good listeners about the ups and downs of this project. Lastly, I appreciate the partial support from NSF RTG grant 0838212.

\bibliographystyle{amsalpha}
\bibliography{eqd}

\end{document}